\documentclass[10pt]{article}

\usepackage{amssymb,latexsym}
\usepackage{amsmath,amscd}
\usepackage{theorem}
\usepackage[all]{xy}
\usepackage{enumerate}

\newcommand{\bb}[1]{\mathbb{#1}}
\newcommand{\cl}[1]{\mathcal{#1}}
\newcommand{\mr}[1]{\mathrm{#1}}

\newcommand{\ra}{\rightarrow}

\newcommand{\map}{\longrightarrow}

\newcommand{\bs}{\backslash}
\newcommand{\ep}{\hfill $\square$} 

\newtheorem{lemma} {Lemma} [section]
\newtheorem{proposition} [lemma] {Proposition}

\newtheorem{theorem} [lemma] {Theorem}
\newtheorem{corollary} [lemma] {Corollary}
\newtheorem{definition}[lemma] {Definition}

\theorembodyfont{\rm} 

\newtheorem{remark}[lemma]{Remark}

\newenvironment{proof}{{\sc Proof:}}{{\hspace*{\fill} $\square$\\}}

\numberwithin{equation}{section}

\title{\bf Bases for Riemann-Roch spaces of one point divisors on an
optimal tower of function fields}
\date{}

\author{Francesco Noseda\thanks{Instituto de Matem\'atica, UFRJ, 
CP 68530, Cidade Universit\'aria, CEP: 21941-909, Rio de Janeiro/RJ - Brasil
(noseda@im.ufrj.br, luciane@im.ufrj.br).}\and 
Gilvan Oliveira\thanks{Departamento de Matem\'atica, CCE/UFES, 
 Goiabeiras, CEP: 29075-910,
Vit\'oria/ES - Brasil (jgilvanol@gmail.com).}\and
Luciane Quoos\footnotemark[1]}

\begin{document}

\maketitle

\begin{abstract}
\noindent
For applications in algebraic geometric codes, an
explicit description of bases of 
 Riemann-Roch spaces of divisors on function fields
 over finite fields is needed.
We give an algorithm to compute
such bases for one point divisors, and Weierstrass semigroups over an optimal tower of function fields.
We also explicitly compute Weierstrass semigroups till level eight.
\end{abstract}

\noindent
MSC2010 Subject Classification Numbers: 14H05, 14H55, 14G50, 94B05.


\section{Introduction}

Algebraic geometric codes are defined by means of
 Riemann-Roch spaces of divisors on function fields
 over finite fields.
In practice, for applications in coding theory, one needs an
explicit description of bases of such spaces. 
 The problem of computing these bases is a hard one, and it is addressed, for instance, in
 \cite{Maharaj2}, \cite{MH}, \cite{PST}, \cite{Shum} and \cite{VH}. 
In this correspondence we give an algorithm to compute
bases and Weierstrass semigroups over an optimal tower of function fields.

We use the language of function fields as in \cite{Stichtenoth}.
Let $\bb{F}_q $ be the finite field with $q$ elements, 
and consider a tower of function fields $\mathcal{T}=(T_j)_{j \geq 0}$ 
over $\bb{F}_q $. Let $g(T_j)$ denote the genus of $T_j/\bb{F}_q$ and 
$N(T_j)$ its number of $\bb{F}_q $-rational points. 
The well-known Drinfeld-Vl\u{a}dut bound (see $\cite{DV}$) guarantees that the limit
$ \lambda(\mathcal{T}) := \lim_{j \rightarrow \infty} N(T_j)/g(T_j)$
satisfies the inequality:
$$ 0 \leq \lambda(\mathcal{T})  \leq \sqrt q - 1.$$

The tower $\mathcal{T}$ is said to be $\it{good}$ if $\lambda(\mathcal{T})>0$, and 
$\it{optimal}$ if $\lambda(\mathcal{T})=\sqrt{q}-1$.
In the literature there are many explicit descriptions of several good and optimal towers 
(see  \cite{GB}, \cite{GSInv},
\cite{GSExplicit} and \cite{CMSE}),  from which
algebraic geometric codes can be constructed 
having parameters that attain the Tsfasman-Vl\u{a}dut-Zink bound (see \cite{Stichtenoth}). 

In this paper we consider the tower $\mathcal{T}=(T_j)_{j\geq 0}$ 
over the finite field $\bb{F}_{p^2} $ in odd characteristic. 
This tower is defined recursively by $T_0=\bb{F}_{p^2}(x_0)$ and, 
for $j \geq 0, \, T_{j+1}=T_j(x_{j+1})$, where the function $x_{j+1}$ satisfies 
the relation: 
$$x_{j+1}^2=\frac{x_j^2+1}{2x_j}.$$ 
 This tower was introduced and proven to be optimal in \cite{GSTame}. Let
$P_\infty^j$ be the unique pole of the function $x_0$ in $T_j$. 
For each $s \in \mathbb{N}$ fix the divisors $sP_\infty^j$, and let
$D^j=P_1 + \cdots + P_{N}$ be 
the sum of all rational places in $T_j$ that lie over a set $\Omega$ of points in
$T_0$ that completely
split in the tower, see \cite{GSTame}. We get a sequence of algebraic geometric codes defined, 
for $j \geq 0$,  as:
$$\mathcal{C}_{L}(sP_\infty^j, D^j)=
\{ (z(P_1), \cdots , z(P_{N})) \, 
|\, z \in L(sP_\infty^j) \} \subseteq \mathbb{F}_{p^2}^{N}$$
where the Riemann-Roch space $L(sP_\infty^j)$ is defined by:
$$L(sP_\infty^j)=\{ z \in T_j \,|\, \text{the divisor of } z \mbox{ satisfies } 
(z) \geq -sP_\infty^j \}.$$
The knowledge of  explicit bases of the 
Riemann-Roch spaces allows to construct the matrices of such codes. The main 
result of this paper, Theorem \ref{ThBasisLPinf}, is an algorithm to compute such bases.  
The central idea is to apply results of $\cite{Maharaj}$ (see Theorem \ref{ThMaharaj} in
Section \ref{SecPrel} below)
to decompose the vector 
space $L(sP_\infty^j)$ in $T_j$ as a direct sum of Riemann-Roch spaces of divisors 
at the lower level $T_{j-1}$, and continue this way till the rational function field
$T_0$, where the bases can be easily computed. In order this process to be performed, 
the divisors we get at each level $k<j$ should be invariant for the action of the Galois
group of $T_k/T_{k-1}$. Unfortunately this condition is not always satisfied and
the procedure has to be suitably modified, as done in Sections
\ref{SubsecBasisLPinfxjres1} and \ref{SubsecII}. 
 
As a consequence of the main result we get an algorithm to compute the 
Weierstrass semigroups:
$$H(P_\infty^j)=\{ s \in \mathbb{N} \, |\, \exists \,\, z \in T_j 
\text{ s.t. the pole divisor of } z \mbox{ satisfies } 
(z)_\infty = sP_\infty^j  \}$$
at the totally ramified points $P_\infty^j$, Theorem \ref{ThWeierLPinf}.
As an application, we also explicitly present the semigroups till level eight,
Section \ref{SecSemigroup}.

The ramification structure of the tower and
the computation of the genus is presented in Section \ref{SecPrel}.  \\

\noindent
\textbf{Acknowledgments.} The authors would like to express their gratitude to A. Garcia
for suggesting the subject and for valuable discussions, and to 
P. Zingano for useful help in the computational part.

\section{Preliminaries}\label{SecPrel}

The object of study is an asymptotically optimal tower of functions fields
defined over the finite field $K=\bb{F}_{p^2}$ 
with $p^2$ elements, where $p$ is an odd prime. This tower  is recursively
defined by: $T_0=K(x_0)$ and, 
for $j \geq 0, \, T_{j+1}=T_j(x_{j+1})$, where the function $x_{j+1}$ satisfies 
the relation: 
$$x_{j+1}^2=\frac{x_j^2+1}{2x_j}.$$
 The field $K$ is the full constant field of $T_j/K$ for any $j\geq 0$. Moreover,
 the extension $T_{j+1}/T_j$ is a Kummer extension of degree $2$. 
For the general theory of function fields we refer
to \cite{Stichtenoth}.

\vspace{5mm}

 For $\alpha\in K\cup \{\infty\}$ we denote by
$ P^0_\alpha$
the unique place of $T_0/K$ such that $x_0(P^0_\alpha)=\alpha$.
We start by studying the ramification structure of the tower above the places $P^0_\alpha$
with $\alpha\in\{\infty,0,\pm i,\pm 1\}$, where $i^2=-1$; this is the content of
Lemma \ref{LemmaRamification}. We shall consider the results of the
lemma, and its consequences, to be known.
Nonetheless, we are not aware of a complete written account of the
ramification structure, and the results we need are beyond what is explicitly written
about it in \cite{GSTame}, \cite{GSExplicit}, and \cite{Stichtenoth}.

\begin{lemma} \label{RemRamification}
{\itshape
For all $j\geq 0$, let $P$ be a place of $T_j/K$ and $Q$ be
a place of $T_{j+1}/K$ that lies over $P$. Then: \\

\noindent
(a) We have:
$$ e(Q|P)=\left\{
\begin{array}{ll}
2&\quad\mbox{if}\quad v_P\left(\frac{x_j^2+1}{2x_j}\right)\quad\mbox{is odd}\\\noalign{\medskip}
1&\quad\mbox{if}\quad v_P\left(\frac{x_j^2+1}{2x_j}\right)\quad\mbox{is even}.
\end{array}\right.$$

\noindent
(b) We have:
$$\begin{array}{lcl}
x_{j+1}(Q)=\infty &\qquad\iff\qquad& x_{j}(P)\in\{\infty,0\}
\\\noalign{\medskip}
x_{j+1}(Q)\in\{1,-1\} &\qquad\iff\qquad &x_{j}(P)=1.
\\\noalign{\medskip}
x_{j+1}(Q)=0 &\qquad\iff\qquad& x_{j}(P)\in\{i,-i\}
\\\noalign{\medskip}
x_{j+1}(Q)\in\{i,-i\} &\qquad\iff\qquad &x_{j}(P)=-1.
\end{array}$$

\noindent
(c) If $x_{j}(P)\in\{1,-1\}$ then $e(Q|P)=f(Q|P)=1$.
In both cases
there are exactly two places in $T_{j+1}/K$ over $P$, and $x_{j+1}$
takes one of the values admitted by item $(b)$ in one place, and the other
value in the other place.\\
}
\end{lemma}

\begin{proof} It follows from the theory of Kummer extensions of function fields, see
\cite{Stichtenoth}.

\end{proof}

\vspace{0mm}

\begin{lemma}\label{LemmaRamification}
{\itshape Denote $\cl{R}:=\{P^0_\infty, P^0_0, P^0_i, P^0_{-i}, P^0_1,P^0_{-1}\}$.
Then, for any $j\geq 0$ the following holds. \\

\noindent
(a) For $\alpha\in\{\infty,0,\pm i\}$ there is a unique place $P^j_\alpha$
in
$T_j/K$ over $P^0_\alpha$. If $j\geq 1$ then
$e(P^{j}_\alpha|P^{j-1}_\alpha)=2$.
Moreover $v_{P^j_\infty}(x_j)=-1$ and:
$$ v_{P^j_0}(x_j)=\left\{
\begin{array}{ll}
1&\quad\mbox{if}\quad j=0\\\noalign{\medskip}
-1&\quad\mbox{if}\quad j\geq 1
\end{array}\right.\qquad\qquad
v_{P^j_{\pm i}}(x_j)=\left\{
\begin{array}{ll}
0&\quad\mbox{if}\quad j=0\\\noalign{\medskip}
1&\quad\mbox{if}\quad j= 1 \\\noalign{\medskip}
-1&\quad\mbox{if}\quad j\geq 2.
\end{array}\right.$$

\noindent
(b) For $\beta\in\{\pm 1\}$ there is a unique place $P^j_\beta$
in $T_j/K$ with $x_j(P^j_\beta)=\beta$. If $j\geq 1$, we have:
$P^j_{\pm 1}|P^{j-1}_1$ and
$e(P^j_{\pm 1}|P^{j-1}_1)=1$.  \\

\noindent
(c) Let $r$ with $0\leq r\leq \lfloor(j-3)/2 \rfloor$; notice
that $j\geq 3$ and $r<j-1$. Denote by
$Q=Q^j_r$ a place
in $T_j/K$
that lie over $P^r_{-1}$. Then the sum of the degrees of such places is
$2^{r+2}$. Moreover, $
e(Q^{j}_r|Q^{j-1}_r)=2$ and $v_Q(x_j)=-1$.\\

\noindent
(d) Let $r$ with $\lfloor(j-1)/2 \rfloor \leq r\leq  j-2$;
notice that $j\geq 2$ and $r\geq 0$. Denote by
$Q=Q^j_r$ a place
in $T_j/K$
that lie over $P^r_{-1}$. Then the sum of the degrees of such places is
$2^{j-r}$. Moreover,
we have $e(Q^{j}_r|Q^{j-1}_r)=1$ and:
$$v_Q(x_j)=\left\{
\begin{array}{ll}
 -2^{2r-j+2} &\quad\mbox{if}\quad r\leq j-3\\\noalign{\medskip}
 2^{j-2} &\quad\mbox{if}\quad r= j-2.
\end{array}\right.$$

\noindent
(e) For $\beta\in\{\pm i\}$ there is a unique place $Q^j_\beta$
in $T_j/K$ with $x_j(Q^j_\beta)=\beta$. If $j\geq 1$, we have:
$Q^j_{\pm i}|P^{j-1}_{-1}$ and
$e(Q^j_{\pm i}|P^{j-1}_{-1})=1$, and $Q^j_{\pm i}$
are the unique places of $T_j/K$ that lie over $P^{j-1}_{-1}$.\\

\noindent
(f) The places of $T_j/K$ defined in items $(a)$ to $(e)$ are the unique places
over $\cl{R}$. With the exception of $Q^0_{\pm i}=P^0_{\pm i}$, they are all distinct.\\

\noindent
(g) The function $x_j$ has no zeros nor poles in $T_j$ outside the set of places over
$\cl{R}$.\\

\noindent
(h) The places $Q^j_{\pm i}$ defined in item $(e)$ satisfy $v_{Q^j_{\pm i}}(x_j^2+1)=2^j$.
}
\end{lemma}

\begin{proof} The proof is by induction on $j$. For $j=0$ the thesis follows by the properties
of the rational function field. 
Let $j\geq 0$, and assume the thesis is true for $j$.\\

(a) Use Lemma \ref{RemRamification} to deduce $v_{P^{j+1}_{\infty,0}}(x_{j+1})=-1$
and $e(P^{j+1}_{\infty,0}|P^{j}_{\infty,0})=2$ for $j\geq 0$, and to deduce
$v_{P^{j+1}_{\pm i}}(x_{j+1})=-1$
and $e(P^{j+1}_{\pm i}|P^{j}_{\pm i})=2$ for $j\geq 1$. For $j=0$ the assertions
$v_{P^{1}_{\pm i}}(x_{1})=1$
and $e(P^{1}_{\pm i}|P^{0}_{\pm i})=2$ follows by
$v_{P^0_{\pm i}}(x_0^2+1)=1$ and by the same Lemma.
 \\

(b) It follows by  Lemma \ref{RemRamification}.\\

(c) Let $r$ be such that $0\leq r\leq \lfloor (j-2)/2\rfloor$, which implies
$j\geq 2$ and $r<j$. Let $Q$ be a place in $T_{j+1}/K$ that lies
over $P_{-1}^r$. The image $P$ of $Q$ in $T_j/K$ lies over $P_{-1}^r$ as
well. Applying
Lemma \ref{RemRamification}, the thesis follows 
if we prove the claim that: $v_P(x_j)=-1$, and the sum
of the degrees of such places $P$ is $2^{r+2}$.

First, if $r\leq \lfloor (j-3)/2\rfloor$, then the claim follows by induction
hypothesis $(c)$.

Second, if $r> \lfloor (j-3)/2\rfloor$ then
$j$ is even and $r=(j-2)/2$, i.e., $j=2r+2$. Notice that $r
=\lfloor (j-1)/2\rfloor\leq j-2$. Then the result follows
by induction hypothesis $(d)$.\\

\noindent
(d) Let $r$ with $\lfloor j/2 \rfloor \leq r\leq  j-1$,
which implies $j\geq 1$ and $r\geq 0$. Let $Q$ be a place in $T_{j+1}/K$ that lies
over $P_{-1}^r$. The image $P$ of $Q$ in $T_j/K$ lies over $P_{-1}^r$ as
well.

First case: $r\leq j-3$. Then $\lfloor(j-1)/2 \rfloor \leq r\leq  j-2$
and we can apply induction $(d)$ to deduce $v_P(x_j)= -2^{2r-j+2}$,
and the sum of the degrees of such places $P$ is $2^{j-r}$.
Since in fact $(j-1)/2\leq r$, then $2r-j+2\geq 1$ and $v_P(x_j)$ is even.
Apply  Lemma \ref{RemRamification}
to deduce $e(Q|P)=1$ and $v_Q(x_{j+1})=-2^{2r-(j+1)+2}$. The claim
on the sum of the degrees follows.

Second case: $r= j-2$. Then $\lfloor(j-1)/2 \rfloor \leq r\leq  j-2$
and we can apply induction $(d)$ to deduce $v_P(x_j)= 2^{j-2}$,
and the sum of the degrees of such places $P$ is $2^{j-r}$.
Since $\lfloor j/2 \rfloor\leq   j-2$, then $j\geq 3$.
Hence $v_P(x_j)$ is even.
Apply Lemma \ref{RemRamification}
to deduce $e(Q|P)=1$ and $v_Q(x_{j+1})=-2^{j-3}=-2^{2r-(j+1)+2}$.
The claim
on the sum of the degrees follows.

Third case: $r=j-1$. In this case,
by induction $(e)$ and $(h)$,
 the place $P$ is one of the places $Q^{j}_{\pm i}$,
and we have: $v_Q(x^2_j+1)=2^j$ and $v_Q(x_j)=0$. By 
Lemma \ref{RemRamification} it follows that $e(Q|P)=1$ and
$v_Q(x_{j+1})=
2^{j-1}=2^{(j+1)-2}$. The result
on the sum of the degrees follows.\\

\noindent
(e) By  Lemma \ref{RemRamification} and
induction $(b)$.\\

\noindent
(f) Any place $Q$ of $T_{j+1}/K$ that lies over $\cl{R}$ lies over a place $P$
of $T_j/K$
that lies over $\cl{R}$. By induction $(f)$, any such a place $P$ is one of
the places of $T_j/K$ defined by items $(a)$ to $(e)$.
The thesis follows by observing that if a place of $T_{j+1}/K$ lies above a place of
$T_j/K$
defined by $(a)$ to $(e)$, then it is itself defined by $(a)$ to $(e)$.\\

\noindent
(g) We have to show that if $Q$ is a place of $T_{j+1}$ that does not lie
over $\cl{R}$ then $v_Q(x_{j+1})=0$. Let $P$ the image of
$Q$ in $T_j$. Since $P$ does not lie over $\cl{R}$, by induction $(g)$ we have
$v_P(x_j)=0$. By the relation between $x_{j+1}$ and $x_j$ it follows that
we have to show that $v_P(x^2_j+1)=0$. Necessarily $v_P(x_j^2+1)\geq 0$. If
it happened that $v_P(x^2_j+1)> 0$ then $x_j(P)=\pm i$ and by induction
$(e)$ the place $P$ would lie over $\cl{R}$.\\

\noindent
(h)
Let's analyse zeros and poles of $x^2_{j+1}+1$ in $T_{j+1}/K$.
The unique zeros are $Q^{j+1}_{\pm i}$, by what was already proven for $(e)$.
Since the two places
$Q^{j+1}_{\pm i}$ are permuted by the action of the unique non-trivial
automorphism of $T_{j+1}/T_j$, then the valuations of $x^2_{j+1}+1$ at these
two places are equal, say equal to $v$. We have to show that $v=2^{j+1}$.
By $(g)$ and $(a)$ to
$(e)$ the poles of $x_{j+1}$ are the following: $P^{j+1}_\infty$, $P^{j+1}_0$ (order $1$);
$P^{j+1}_{\pm i}$ if $j\geq 1$ (order 1);
the places $Q^{j+1}_r$ for $r$ with $0\leq r\leq j-2$
(order $1$ if $r\leq \lfloor (j-2)/2 \rfloor$, the sum of the degrees of
such places is $2^{r+2}$; order
$2^{2r-j+1}$ if $r>\lfloor (j-2)/2 \rfloor$, the sum of the degrees of
such places is $2^{j+1-r}$). We can
compute:
$$ \mr{deg}(x_{j+1})_\infty= 1+\sum_{n=0}^j 2^n=2^{j+1}.$$
Then $ 0=\mr{deg}(x_{j+1}^2+1)=
2v-2\mr{deg}(x_{j+1})_\infty$, and we conclude $v=2^{j+1}$.
\end{proof}

\begin{definition}
For any $j\geq 0$ and for any $r$ with $0\leq r\leq j$,
denote by:
$$D^j_r:= \sum_{Q^j|P^r_{-1}} Q^j$$
the divisor given by the sum of the places in $T_j$ that lie over $P^r_{-1}$.
Notice that $D^j_j=P_{-1}^j$.
Extend the definition of $D^j_r$ for $r=-2,-1$ by 
$D_{-2}^j:=P^j_0$ and $D_{-1}^j:=P^j_i+P^j_{-i}.$

\end{definition}

The following corollary is a consequence of Lemma \ref{LemmaRamification}.

\begin{corollary}\label{PropRamification}
{\itshape
For any $j\geq 0$ and any $r$  with $-2\leq r\leq j$ we have:
$$\mr{deg}(D^j_r)=\left\{
\begin{array}{ll}
 2^{j-r} & \mbox{ if } \quad j\leq 2r+2\\\noalign{\medskip}
  2^{r+2} & \mbox{ if }\quad j\geq 2r+2.
\end{array}
\right.$$
Moreover, if $j\geq 1$ then, for any $r$ with $-2\leq r\leq j$,
 the places in $D^j_r$ ramifies in $T_j/T_{j-1}$ if
and only if $j\geq 2r+3$.\\
}
\end{corollary}

\begin{proposition}\label{PropDivisorxj}
{\itshape The divisors of $x_j$ and $1+x_j$ in $T_j$ are given by the following formulae:\\
$(x_0)=-P_\infty^0+P_0^0$, $
(x_1)=-P_\infty^1-P_0^1+D_{-1}^1$, and:
$$(x_j)={\displaystyle -P_\infty^j-
\sum_{r=-2}^{\lfloor{\frac{j-3}{2}}\rfloor }
 D_r^j-
\sum_{r=\lfloor\frac{j-1}{2}\rfloor}^{j-3}
2^{2r-j+2}  D_r^j+ 2^{j-2} D_{j-2}^j}\quad\mbox{if}\quad j\geq 2,$$
 $ (1+x_0)=-P_\infty^0+P_{-1}^0$, 
$(1+x_1)=-P_\infty^1-P_0^1+2P_{-1}^1$, and:
$$(1+x_j)={\displaystyle -P_\infty^j-
\sum_{r=-2}^{\lfloor{\frac{j-3}{2}}\rfloor }
 D_r^j-
\sum_{r=\lfloor\frac{j-1}{2}\rfloor}^{j-3}
2^{2r-j+2}  D_r^j+ 2^{j} P_{-1}^j}\quad\mbox{if}\quad j\geq 2.$$
}
\end{proposition}

\begin{proof} We apply Lemma \ref{LemmaRamification}. The claim on the divisor of $x_j$ follows. 
The function $1+x_j$
has the same poles as $x_j$; moreover, it has $P_{-1}^j$ as unique zero. 
By the proof of item $(h)$ of the lemma 
we have: $\mr{deg}(1+x_j)_\infty=\mr{deg}(x_j)_\infty=2^j$. Hence, the coefficient
of  $P_{-1}^j$ of $(1+x_j)$ is $2^j$.

\end{proof}

\begin{proposition}\label{PropGenus} {\itshape
For $j\geq 0$ the genus $g_j$ of $T_j/K$ is given by:
$$g_j=\left\{\begin{array}{ll}
(2^{\frac{j+2}{2}}-1)(2^{\frac{j}{2}}-1)&\mbox{if }\,  j\, \mbox{ is
even}\\\noalign{\medskip}
(2^{\frac{j+1}{2}}-1)^2&\mbox{if }\, j\, \mbox{ is odd}.
\end{array}\right.$$
}
\end{proposition}

\begin{proof} By the Riemann-Hurwitz Theorem  we have that
for any $j\geq 0$:
$$ g_{j+1}=2g_j-1+\frac{1}{2}R_j$$
where $R_j=\sum_P \mr{deg}\,P$, and the sum is over all places of $T_j/K$ such that
$v_P((x_j^2+1)/2x_j)$ is odd.
The last condition is satisfied in two cases: 
$v_P(x_j)=0$ and $v_P(x_j^2+1)$ is odd; or $v_P(x_j)$ is odd.
By Lemma \ref{LemmaRamification},
the first case is realized exactly  by $P_{\pm i}^j$ if $j=0$.
The second one is realized by the places: $P^j_\infty$, $P_0^j$; $P^j_{\pm i}$ if $j\geq 1$;
$Q^j_r$ with $0\leq r\leq \lfloor (j-3)/2\rfloor$ (the sum of
the degrees of such places is $2^{r+2}$);
$Q^j_r$ with $\lfloor (j-1)/2\rfloor\leq r\leq j-2$ and: $(a)$ $r=j-2$ and $j-2=0$; or
$(b)$ $r\leq j-3$ and $2r-j+2=0$
(the sum of the degrees of such places is $2^{j-r}$).

If $j$ is odd then cases $(a)$ and $(b)$ above do not contribute, and we have
$ R_j=1+\sum_{n=0}^{{\scriptscriptstyle (j+1)/2}}2^n=2^{\frac{j+3}{2}}$.
If $j$ is even then case $(a)$ contributes with $r=0$ if $j=2$ (the
sum of the degrees is 4),
and case $(b)$ contributes
with $r=(j-2)/2$ if $j\geq 4$ (the sum of the degrees is $2^{r+2}$). Hence,
$ R_j=1+\sum_{n=0}^{{\scriptscriptstyle (j+2)/2}}2^n=2^{\frac{j+4}{2}}$.

From this we deduce the following recursive relations:
$$g_{j+2}=\left\{\begin{array}{ll}
 4g_j+3\cdot2^{\frac{j+2}{2}}-3&\quad\mbox{if}\quad j \quad\mbox{is even}\\\noalign{\medskip}
 4g_j+2^{\frac{j+5}{2}}-3&\quad\mbox{if}\quad j \quad \mbox{is odd}.
\end{array}\right.$$
Since $g_0=0$, from which $g_1=1$, then the thesis follows by induction.
\end{proof}

Next, we state a particular case of Theorem 2.2 \cite{Maharaj}, by H. Maharaj,
that will play a central role in the construction of
a basis of the Riemann-Roch space $L(sP_\infty^j)$. 
In order to state the theorem we need the notions of invariant divisor and 
of restriction of a divisor; this is the content of the next two remarks.

\begin{remark}\label{RemInvDiv}
We will be dealing with divisors in $T_j/K$ of the form:
$$ D=\alpha_\infty P_\infty^j+\sum_{r=-2}^{j-1} \alpha_r D^j_r+\alpha_j P^j_{-1}$$
where the coefficients $\alpha_\infty$, $\alpha_r$ and $\alpha_j$ are integers.
If $j\geq 1$ then such a divisor is invariant under the
action of the Galois group of $T_j/T_{j-1}$ (briefly, invariant in $T_j/T_{j-1}$)
if and only if $\alpha_j=0$. Indeed, because the unique non-trivial automorphism of
$T_j/T_{j-1}$ sends $x_{j}$ to $-x_j$,  the place $P^j_{-1}$ is sent to $P^j_1$,
while the divisors $P_\infty^j$ and $D^j_r$ with $-2\leq r\leq j-1$ are invariant. 

\end{remark}

\begin{remark}\label{RemResDiv}
Let $j\geq 1$, and $D=\alpha_\infty P_\infty^j+\sum_{r=-2}^{j-1} \alpha_r D^j_r$
be a  divisor of $T_j/K$.
Then, applying the definition  given in \cite{Maharaj},
and applying Corollary \ref{PropRamification}, the
restriction of $D$ to $T_{j-1}$ is given by:
$$ D_{|T_{j-1}}=
{\textstyle \left \lfloor \frac{\alpha_\infty}{2}\right \rfloor}P_\infty^{j-1}+
\sum_{r=-2}^{\lfloor{\frac{j-3}{2}}\rfloor }
{\textstyle \left \lfloor \frac{\alpha_r}{2}\right \rfloor}
 D^{j-1}_r+\sum_{r=\lfloor{\frac{j-1}{2}}\rfloor }^{j-1}
 \alpha_r D^{j-1}_r.$$

\end{remark}

\begin{theorem} \label{ThMaharaj}
Let $j\geq 1$ and $D$ be a divisor in $T_j/K$ that is invariant in
$T_j/T_{j-1}$. Then:
$$ L(D)=L(D_{|T_{j-1}})\oplus L([D+(x_j)]_{|T_{j-1}})x_j.$$
\end{theorem}

\begin{proof}
This is Theorem 2.2 of \cite{Maharaj} applied to the
Kummer extension $T_j/T_{j-1}$.
\end{proof}

In the next remark we explain how to compute the power series expansion
of the generators $x_k$ around $P^j_{-1}$, for $k\leq j$, with respect to a suitably chosen
local parameter. This will be used in the sequel to make bases of 
$L(sP_\infty^j)$ and the Weierstrass
semigroups $H(P_\infty^j)$  computable.

\begin{remark}\label{RemPower} The function 
$t:=1-x_0$
is a local parameter around $P^j_{-1}$, for any $j\geq 1$. Indeed, $v_{P^0_1}(1-x_0)=1$ and
$e(P^j_{-1}|P^0_1)=1$. Notice that, for the same reason,
$t$ is a local parameter around
$P^j_1$ as  well, for any $j\geq 0$.

We show how to compute the power series expansion up to any
order $\varepsilon\geq 1$ of $x_j$ around $P_1^j$, for any $j\geq 0$. Assume to have
the expansion:
$ x_j=1+ \sum_{k=1}^\varepsilon a_k t^k + O(t^{\varepsilon+1})$
around $P_1^j$. Then we can compute the expansion
$(x_j^2+1)/2x_j=1+\sum_{k=1}^\varepsilon b_k t^k + O(t^{\varepsilon+1}).$
Since $e(P^{j+1}|P^j)=1$ then the same expansion holds around $P^{j+1}_1$.
Let the expansion of $x_{j+1}$ around $P^{j+1}_1$ be given by:
$ x_{j+1}=1+ \sum_{k=1}^\varepsilon c_k t^k + O(t^{\varepsilon+1})$,
where the $c_k\in K$ have to be computed. By the relation:
$ x_{j+1}^2=(x_j^2+1)/2x_j$
it follows that:
$$1+ \sum_{k=1}^\varepsilon \left(2c_k+\sum_{l=1}^{k-1}c_lc_{k-l}\right)
 t^k + O(t^{\varepsilon+1})=
1+ \sum_{k=1}^\varepsilon b_k t^k + O(t^{\varepsilon+1}).$$
Hence, the unknown coefficients $c_k$ can be computed by induction on
$k$ from
the formulae:
$$ c_k=\frac{1}{2}\left(b_k -\sum_{l=1}^{k-1}c_lc_{k-l}\right) \qquad 1\leq
k\leq \varepsilon.$$

We also remark that for $j\geq 1$ and $k$ with $0\leq k\leq j$ the
expansion of $x_k$ around $P^j_{-1}$ is: the one given above if
$k<j$ (since $e(P^j_{-1}|P^k_1)=1$); minus the one given above if
$j=k$ (since we have to take the other determination of the square root
of $(x_j^2+1)/2x_j$).
\ep

\end{remark}


\section{Hermitian basis of the Riemann-Roch space $L(sP_\infty^j)$}

We state the main result of the paper: Theorem \ref{ThBasisLPinf}. This 
gives bases of the spaces $L(sP^j_\infty)$, as $j$ and $s$ vary, in a
constructive way. As a corollary of the main theorem we get
a constructive way to compute the Weierstrass semigroups $H(P_\infty^j)$: this
is the content of Theorem \ref{ThWeierLPinf}. We recall that a basis of a Riemann-Roch
space is Hermitian (with respect to $P_\infty^j$) 
if its elements have distinct pole order at $P_\infty^j$.

\begin{theorem}\label{ThBasisLPinf}
\textit{There exists an algorithm that, for all $j\geq 0$
and all $m$ with $0\leq m<2^j$,
 constructs integers $c_m^{(j)}$ and functions
$w_m^{(j)}\in T_j$ s.t. for
all $s\in\bb{Z}$ the family parametrized by $m$ and $l$:
$$ x_0^lw_m^{(j)}\qquad\qquad 0\leq m<2^j\qquad 0\leq l\leq
{\textstyle \left \lfloor \frac{s-m}{2^j}\right \rfloor} -c_m^{(j)}:=l_m(s)$$
is a Hermitian basis of $L(sP_\infty^j)$. (When the value of $s$ is such that
$l_m(s)$ is negative for a given $m$, it is understood
that the $m$-th family does not contribute to the basis.)}
\end{theorem}

\begin{proof} The proof is by induction on $j$ and uses Theorems
\ref{ThMaharaj} and \ref{ThBasisLPinfxjres}. For $j=0$, put
$ c^{(0)}_0:=0$ and $ w^{(0)}_0:=1$.

Let $j\geq 1$.  We apply Theorem \ref{ThMaharaj} and get, for
any $s\in\bb{Z}$:
$$L(sP_\infty^j)=L([{sP_\infty^j}]_{|T_{j-1}})
\oplus L([sP_\infty^j+(x_j)]_{|T_{j-1}})x_j.$$
Since
$[sP_\infty^j]_{|T_{j-1}}={\textstyle
\left\lfloor\frac{s}{2}\right\rfloor}P_\infty^{j-1}$
(see Remark \ref{RemResDiv}),
then the induction hypothesis gives a basis of the first direct summand:
$$ x_0^lw_m^{(j-1)}\qquad\qquad 0\leq m< 2^{j-1}\qquad 0\leq l\leq
{\textstyle \left \lfloor \frac{s-2m}{2^j}\right \rfloor} -c_m^{(j-1)}$$
where we used the fact that:
$ {\textstyle\lfloor (\lfloor s/2\rfloor -m)/2^{j-1}\rfloor}=
{\textstyle \left \lfloor (s-2m)/2^j\right \rfloor}$.\\

A basis of $L([sP_\infty^j+(x_j)]_{|T_{j-1}})$ is constructed by
Theorem \ref{ThBasisLPinfxjres}. The desired basis of $L(sP_\infty^j)$ is
obtained by defining, for $0\leq m<2^j$:
$$c_m^{(j)}:=\left\{
\begin{array}{ll}
c_{m/2}^{(j-1)}&\quad\mbox{if }\:m\: \mbox{ is even}\\\noalign{\medskip}
\tilde{c}_{m}^{(j-1)}&\quad\mbox{if }\:m\: \mbox{ is odd}
\end{array}
\right.
\qquad\qquad
w_m^{(j)}:=\left\{
\begin{array}{ll}
w_{m/2}^{(j-1)}&\quad\mbox{if }\:m\: \mbox{ is even}\\\noalign{\medskip}
\tilde{w}_{m}^{(j-1)}x_j&\quad\mbox{if }\:m\: \mbox{ is odd}.
\end{array}
\right.
$$
The fact that the basis is Hermitian follows by Corollary \ref{Cor}.  
\end{proof}

The knowledge of the dimension of $L(sP_\infty^j)$ for fixed $j$ and $s$ variable,
allows to compute the Weierstrass semigroup $H(P_\infty^j)$. 
Because of the structure of the basis
given in the above theorem, we get an elegant and efficient
way to recover the  semigroup.

\begin{theorem}\label{ThWeierLPinf}
\textit{For any $j\geq 0$
the Weierstrass semigroup $H(P_\infty^j)$ can be recovered
 from the coefficients $c_m^{(j)}$ given in Theorem \ref{ThBasisLPinf}
 as follows.
For any $s\in\bb{Z}$, let $q(s)$
and $m(s)$ be quotient and rest of the division of $s$ by $2^j$, i.e.,
$s=2^jq(s)+m(s)$ with $0\leq m(s)< 2^j$. Then:
$$ s\in H(P_\infty^j)\qquad\iff\qquad q(s)\geq c_{m(s)}^{(j)}.$$
}
\end{theorem}

\begin{proof} Fix $j\geq 0$.
We define a function $\sharp:\bb{Z}\ra\bb{Z}$ by:
$ \sharp(L):=L+1$ if $ L\geq -1$, and $\sharp(L):=0$ if $L\leq -1$.
By counting the elements of the basis of $L(sP_\infty^j)$ given in Theorem
\ref{ThBasisLPinf} we get that for any $s\in\bb{Z}$:
$ \mr{dim}\,L(sP_\infty^j)=\sum_{m=0}^{2^j-1}\sharp(l_m(s))$. 
Define $\Delta(s):=\mr{dim}\,L(sP_\infty^j)-\mr{dim}\,L((s-1)P_\infty^j)=$
$\sum_{m=0}^{2^j-1}[\sharp(l_m(s))-\sharp(l_m(s-1))].$
By writing $s-m=2^jq+r$ with $0\leq r<2^j$, we can deduce that for all $m$ and $s$:
$$ l_m(s)-l_m(s-1)={\textstyle \left \lfloor \frac{s-m}{2^j}\right \rfloor}-
{\textstyle \left \lfloor \frac{s-1-m}{2^j}\right \rfloor}=
\left\{
\begin{array}{ll}
1&\quad\mbox{if }\quad m=s\:\mbox{ mod }2^j\\\noalign{\medskip}
0&\quad\mbox{if }\quad m\neq s\:\mbox{ mod }2^j.
\end{array}\right.$$
It follows that
$\Delta(s)=
\sharp(l_{m(s)}(s))-\sharp(l_{m(s)}(s-1))$.
Then, for any $s\in\bb{Z}$:
$$ s\in H(P^j_\infty)\quad\iff\quad \Delta(s)=1\quad\iff\quad l_{m(s)}(s)\geq 0
\quad \iff \quad q(s)-c_{m(s)}^{(j)}\geq 0.$$
\end{proof}

\begin{corollary}\label{Cor}
With the same notation of Theorem \ref{ThBasisLPinf} we have that the pole divisor
of the function $w_m^{(j)}$ is:
$$ (w^{(j)}_m)_\infty= (2^jc^{(j)}_m+m)P^j_\infty.$$
Moreover, the set:
$$ \{2^j\}\cup \{2^jc^{(j)}_m+m\,| \, 0< m< 2^j\} $$ 
generates $H(P_\infty^j)$. 
\end{corollary}

\begin{proof}
Fix $m$ and define $s:=2^j c^{(j)}_m+m$. Since $l_m(s-1)=-1$ and $l_m(s)=0$ then
$w_m^{(j)}\in L(sP^j_\infty)\bs L((s-1)P_\infty^j)$.
The function
$x_0$ has a unique pole of order $2^j$ at $P^j_\infty$, then we are done by
Theorem \ref{ThWeierLPinf}. 
(We remark that this is not a minimal set of generators in general.) 
\end{proof}

Thanks to Theorem \ref{ThMaharaj}, the construction of a basis
of $L(sP^j_\infty)$ was reduced, in the proof of Theorem
\ref{ThBasisLPinf}, to the construction of a basis of $ L([sP_\infty^j+(x_j)]_{|T_{j-1}})$.
We state the main technical result of the paper.

\begin{theorem}\label{ThBasisLPinfxjres}
{\itshape There exists an algorithm  that, for all $j\geq 1$
and all odd $m$ with $0<m<2^{j}$,
 constructs integers $\tilde{c}_m^{(j-1)}$ and functions
$\tilde{w}_m^{(j-1)}\in T_{j-1}$ s.t. for
all $s\in\bb{Z}$ the family parametrized by $m$ and $l$:
$$ x_0^l\tilde{w}_m^{(j-1)}\qquad\qquad \mbox{ m odd }\quad 0< m< 2^j\qquad 0\leq l\leq
{\textstyle \left \lfloor \frac{s-m}{2^{j}}\right \rfloor} -\tilde{c}_m^{(j-1)}$$
is a basis of $L([sP_\infty^j+(x_j)]_{|T_{j-1}})$.}
\end{theorem}

\begin{proof} The result is a direct consequence of Proposition \ref{PropDownUp}.
Item $(ii)$ applied with $k=j-1$ gives a basis of 
$ L(A^{j-1}_1(s)):= L([sP_\infty^j+(x_j)]_{|T_{j-1}}).$
Thanks to item $(i)$ we can reorder the families
in such a way  to get the basis in the stated form.

\end{proof}


\subsection{Basis of $L([sP_\infty^j+(x_j)]_{|T_{j-1}})$: first part.}
\label{SubsecBasisLPinfxjres1}

This subsection and the next one are dedicated to the proof of
Theorem \ref{ThBasisLPinfxjres}. We fix once and for all an integer $j\geq 1$.
All of what will be defined in the sequel will depend on $j$ but
we will not indicate this fact in the notation.

In this subsection we will construct
divisors $A^{k}_n(s)$ and $B^k_n(s)$ at level $k\leq j-1$.
 In the next subsection suitable bases of the Riemann-Roch spaces of
these divisors will be constructed.

\begin{proposition}\label{PropUpDown}
{\itshape There is an algorithm to construct
divisors $A^{k}_n(s)$ and $B^k_n(s)$, $s\in\bb{Z}$,
and integers $ a_n^k$, $b_n^k$, $\alpha_n^k$, $\gamma_n^k$, and $\delta_n^k $,
labeled by integers $k$ and $n$ s.t. $  0\leq k\leq j-1$ and $1\leq n\leq 2^{j-k-1} $,
in such a way that the following conditions are satisfied.
\begin{enumerate}[(i)]
\item $A^{j-1}_1(s)=[sP_\infty^j+(x_j)]_{|T_{j-1}}$.
\item The divisors $A^k_n(s)$ and $B^k_n(s)$
have the form:
$$ \begin{array}{lll}
(a)\quad &A^{k}_n(s)=&
{\textstyle
\left\lfloor\frac{s+a^{k}_n}{2^{j-k}}\right\rfloor}P_\infty^{k}
+\sum_{m=-2}^{k-1}
\alpha_{n,m}^{k}D_m^{k}+\alpha_n^{k}P_{-1}^{k}\\\noalign{\bigskip}
(b)\quad &B^{k}_n(s)=&
{\textstyle
\left\lfloor\frac{s+b^{k}_n}{2^{j-k}}\right\rfloor}P_\infty^{k}
+\sum_{m=-2}^{k-1}
\beta_{n,m}^{k}D_m^{k}.
\end{array}$$
\item $ B^{k}_n(s)= A^{k}_{n}(s)
+\gamma^{k}_{n}(1+x_{k})
+\delta^{k}_{n}P^{k}_{-1}$.
\item If $k<j-1$, and  $\lceil\cdot\rceil$ is the roof function, then:
$$ A^{k}_n(s)=\left\{
\begin{array}{ll}
{\textstyle [B^{k+1}_{\lceil n/2 \rceil}(s)]_{|T_{k}}}
&\mbox{ if }n\mbox{ is odd}\\\noalign{\medskip}
{\textstyle [B^{k+1}_{\lceil n/2 \rceil}(s)+(x_{k+1})]_{|T_{k}}}
&\mbox{ if }n\mbox{ is even.}
\end{array}\right.$$
\item $ -\alpha^k_n=\gamma_n^k2^k+\delta_n^k$ and $ 0\leq \delta^k_n<2^k$.
\item For $k$ fixed,
all the odd integers modulo
$2^{j-k}$ appear exactly once: (a) in the sequence of the coefficients $a^{k}_n$, as $n$ varies, and
(b) in the sequence of the coefficients $b^k_n$, as $n$ varies.
\end{enumerate}
}
\end{proposition}

\begin{proof} The proof is by descending induction on $k$, and it is
given in three steps. In step 1 we define the divisors $A^{j-1}$. In step 2, for any $k$,
we
define the divisors $B^k$ given the divisors $A^k$. In step 3,
for any $k<j-1$, we define the divisors $A^k$ given the divisors $B^{k+1}$.\\

\noindent
\texttt{Step 1.} For $k=j-1$ 
define, for all $s\in\bb{Z}$, $ A^{j-1}_1(s):=[sP_\infty^j+(x_j)]_{|T_{j-1}}$,
so that $(i)$ is satisfied by definition. Define the
coefficients $a_1^{j-1}:=-1$ and $\alpha^{j-1}_1:=0$.
Since the divisor $(x_j)$ at level $j$
has coefficient $-1$ at $P^j_\infty$, see  
Proposition \ref{PropDivisorxj}, then
$(ii,a)$ is satisfied.
Moreover, $(vi,a)$ is trivially satisfied.
\\

\noindent
\texttt{Step 2.} Fix $k$ with $0\leq k\leq j-1$,
and assume to have defined divisors $A^k_n(s)$
and integers $a_n^k$ and $\alpha^k_n$ s.t.
$(ii,a)$ and $(vi,a)$ are satisfied for the given
$k$. Define $\gamma^k_n$ and $\delta^k_n$ by dividing $-\alpha^k_n$ by
$2^k$, so that $(v)$ is satisfied for the given $k$ by definition.
Define $B^k_n(s)$ by $(iii)$, so that $(iii)$ is satisfied for the given
$k$ by definition.
We will define coefficients $b_n^k$ such that $(ii,b)$ and $(vi,b)$ are
 satisfied for the given $k$.

By the form of the divisor $(1+x_k)$, see 
Proposition \ref{PropDivisorxj}, item $(ii,b)$
 follows by
taking care of the coefficients of $B^k_n(s)$ at $P^k_\infty$ and $P^k_{-1}$.
Since the coefficient of $(1+x_k)$ at $P^k_{-1}$ is $2^k$ then
by $(v)$ it is easily seen that
the coefficient of $B^k_n(s)$ at $P^k_{-1}$ is zero as desired.
Since the coefficient of $(1+x_k)$ at $P^k_{\infty}$ is $-1$ then
the coefficient of $B^k_n(s)$ at $P^k_{\infty}$ is given by 
$\lfloor (s+a^{k}_n-2^{j-k}\gamma_n^k)/2^{j-k}\rfloor$.
We define $ b_n^k:=a^{k}_n-2^{j-k}\gamma_n^k$
and this concludes with items $(ii,b)$ and $(vi,b)$ for the given $k$,
the latter because
$b_n^k=a_n^k$ mod $2^{j-k}$.
\\

\noindent
\texttt{Step 3.}
Fix $k$ with $0\leq k<j-1$, and assume to have defined divisors $B^{k+1}_n(s)$
and integers $b_n^{k+1}$ s.t.
$(ii,b)$ and $(vi,b)$ are satisfied for $k+1$ in place of $k$.
We will define divisors
$A^k_n(s)$ and integers $a_n^k$ and $\alpha^k_n$
s.t. $(ii,a)$ and $(vi, a)$ are satisfied
for $k$.

Define $A^k_n(s)$ by $(iv)$, so that $(iv)$ is satisfied for $k$.
Let:
$$ a^{k}_n:=\left\{
\begin{array}{ll}
{\textstyle b^{k+1}_{\lceil n/2 \rceil}}
&\mbox{ if }n\mbox{ is odd}\\\noalign{\medskip}
{\textstyle b^{k+1}_{\lceil n/2 \rceil}}-2^{j-k-1}
&\mbox{ if }n\mbox{ is even.}
\end{array}\right.$$
By Proposition \ref{PropDivisorxj},
the coefficient of the divisor $(x_{k+1})$ at $P^{k+1}_{\infty}$ is equal to
$-1$, then $(ii,a)$ is satisfied for $k$. In order to prove $(vi,a)$ we argue as follows. 
For $n=1,...,2^{j-k-1}$ with
$n$ odd, the number $\lceil n/2 \rceil$ takes all the values $1,...,2^{j-(k+1)-1}$
exactly once. The same is true taking the even values of $n$.
It follows that the set $\{a^k_n\,|\,n=1,...,2^{j-k}\}$
coincides with the set: $\{b^{k+1}_n, b^{k+1}_n-2^{j-k-1}\,|\,
n=1,...,2^{j-(k+1)-1}\} $. Since $(vi,b)$ holds with $k+1$ in place of $k$ by
hypothesis, then the $b^{k+1}_n$'s take
exactly once all the
odd residues modulo $2^{j-(k+1)}$. The conclusion $(vi,a)$ for the given $k$
follows.
\end{proof}

The next corollary
will be used in the next subsection to construct bases of the divisors
given in the above proposition.

\begin{corollary}\label{CorBfromA}
{\itshape
For all integers $k,n,s$ with $0\leq k< j-1$ and $1\leq n\leq2^{j-k-2}$
we have:
$$ L(B^{k+1}_{n}(s))=L(A^{k}_{2n-1}(s))\oplus
L(A^{k}_{2n}(s))x_{k+1}.$$
}
\end{corollary}

\begin{proof} It follows by item $(ii,b)$
of Proposition \ref{PropUpDown} (applied with $k+1$ in place of $k$)  that
$B^{k+1}_n(s)$ is invariant in $T_{k+1}/T_{k}$. The thesis of
the corollary follows by
Theorem \ref{ThMaharaj} and item $(iv)$ of Proposition \ref{PropUpDown}.
\end{proof}

\subsection{Basis of $L([sP_\infty^j+(x_j)]_{|T_{j-1}})$: second part.}\label{SubsecII}

In this subsection we will complete the proof of Theorem \ref{ThBasisLPinfxjres}
by computing bases of the Riemann-Roch spaces
$L(A^{k}_n(s))$ and $L(B^{k}_n(s))$ of the divisors constructed in
Subsection \ref{SubsecBasisLPinfxjres1}. The integer $ j\geq 1$
will be fixed once and for all.
\\

\begin{proposition}\label{PropDownUp}
{\itshape There is an algorithm to construct
functions $ w_m^k$, $z_m^k\in T_k$
and integers $ b_m$, $c_m^k$, $d_m^k$
labeled by integers $k$ and $m$ s.t. $  0\leq k\leq j-1$ and $1\leq m\leq 2^{j-1} $,
in such a way that the following conditions are satisfied.
\begin{enumerate}[(i)]
\item For all $m$,  $ 0\leq b_m<2^j$.
Moreover, as $m$ varies, the coefficients $b_m$
take exactly once all the odd values modulo $2^j$.
\item For all $k$, for all $n$ with $1\leq n\leq 2^{j-k-1}$,
and for all $s\in\bb{Z}$, the family parametrized
by $m$ and $l$:
$$ x_0^lw_m^k\qquad\qquad (n-1)2^k+1\leq m\leq n2^k\qquad
0\leq l\leq
{\textstyle \left \lfloor \frac{s-b_m}{2^{j}}\right \rfloor}
-c_m^k$$
is a basis of $L(A^{k}_n(s))$.
\item For all $k$, for all $n$ with $1\leq n\leq 2^{j-k-1}$,
and for all $s\in\bb{Z}$, the family parametrized
by $m$ and $l$:
$$ x_0^lz_m^k\qquad\qquad (n-1)2^k+1\leq m\leq n2^k\qquad
0\leq l\leq
{\textstyle \left \lfloor \frac{s-b_m}{2^{j}}\right \rfloor}
-d_m^k$$
is a basis of $L(B^{k}_n(s))$.
\end{enumerate}
}
\end{proposition}

\begin{proof}  The proof is by induction on $k$, and it is
given in three steps. In step 1 we construct bases of the spaces $L(B^0)$.
In step 2, for any $k$,
we
construct bases of the spaces $L(A^k)$ given bases of the spaces
$L(B^k)$. In step 3,
for any $k<j-1$, we construct bases of the spaces $L(B^{k+1})$
given bases of the spaces $L(A^{k})$.\\

\noindent
\texttt{Step 1.} 
For $k=0$, and $n=1,...,2^{j-1}$, and $s\in\bb{Z}$ we have,
by item $(ii,b)$ of Proposition \ref{PropUpDown}:
$$ B^{0}_n(s)=
{\textstyle
\left\lfloor\frac{s+b^{0}_n}{2^j}\right\rfloor}P_\infty^{0}
+\beta_{n,-2}^0P_0^0+\beta_{n,-1}^0D_{-1}^0.$$
Define, for $m=1,...,2^{j-1}$, $ z^0_m:= x_0^{-\beta_{m,-2}^0}(1+x_0^2)^{-\beta_{m,-1}^0}$.
Define $b_m$ by a division by $2^j$, that is: $-b^0_m=q_m2^j+b_m$ and $ 0\leq b_m<2^j$,
and define  $d^0_m:=q_m-\beta_{m,-2}^0-2\beta_{m,-1}^0$.
Then item $(i)$ is satisfied thanks to
item $(vi,b)$ of Proposition \ref{PropUpDown} for $k=0$.
Since $T_0$ is the rational function field, then
 the family parametrized by $l$:
$$x_0^lz_m^0=x_0^{l-\beta_{m,-2}^0}(1+x_0^2)^{-\beta_{m,-1}^0}
\qquad 0\leq l\leq
{\scriptstyle \left \lfloor \frac{s+b_m^0}{2^{j}}\right \rfloor}
+\beta_{m,-2}^0+2\beta_{m,-1}^0
={\textstyle \left \lfloor \frac{s-b_m}{2^{j}}\right \rfloor}
-d_m^0$$
is a basis of $L(B^{0}_n(s))$. Then $(iii)$ is satisfied for $k=0$.\\

\noindent
\texttt{Step 2.} Fix $k$ with $0\leq k\leq j-1$.
Assume that, for $m=1,...,2^{j-1}$, the elements
$d^{k}_m$ and $z^{k}_m$ have already been defined
in such a way that item $(iii)$ is satisfied for the given $k$. We will define,
for $m=1,...,2^{j-1}$, elements
$c^{k}_m$ and $w^{k}_m$
in such a way that item $(ii)$ is satisfied for $k$.

Fix $n$ with $1\leq n\leq 2^{j-k-1}$.
For any $s\in\bb{Z}$, let $C^k_n(s):=A^k_n(s)+\gamma^k_n(1+x_k)$.
Define $c_m^k:=d_{m,\delta}$ and $ \tilde{w}_n^k:=z_{m,\delta}$, for
$m=(n-1)2^k+1,..., n2^k$,
where $d_{m,\delta}$ and $z_{m,\delta}$ are constructed in
Lemma \ref{LemmaInductionEpsilon}.
Then, for any $s\in\bb{Z}$ the family parametrized
by $m$ and $l$:
$$ x_0^l\tilde{w}_m^k\qquad\qquad (n-1)2^k+1\leq m\leq n2^k\qquad
0\leq l\leq
{\textstyle \left \lfloor \frac{s-b_m}{2^{j}}\right \rfloor}
-c_m^k$$
is a basis of $L(C^{k}_n(s))$. Then, defining $w^k_m:=\tilde{w}^k_m(1+x_k)^{\gamma^k_n}$,
the family parametrized
by $m$ and $l$:
$$ x_0^lw_m^k\qquad\qquad (n-1)2^k+1\leq m\leq n2^k\qquad
0\leq l\leq
{\textstyle \left \lfloor \frac{s-b_m}{2^{j}}\right \rfloor}
-c_m^k$$
is a basis of $L(A^{k}_n(s))$ for any $s\in\bb{Z}$, as desired.
\\

\noindent
\texttt{Step 3.} Fix $k$ with  $0\leq k<j-1$.
Assume that, for $m=1,...,2^{j-1}$, the elements
$c^{k}_m$ and $w^{k}_m$ have already been defined
in such a way that item $(ii)$ is satisfied for the given $k$. We will define,
for $m=1,...,2^{j-1}$, elements
$d^{k+1}_m$ and $z^{k+1}_m$
in such a way that item $(iii)$ is satisfied with $k+1$ in place of $k$.

Define $ d_m^{k+1}:=c_m^k$.
For $n=1,...,2^{j-k-2}$, and $m=(n-1)2^{k+1}+1,...,n2^{k+1}$, denote:
$$z^{k+1}_m:=\left\{\begin{array}{ll}
w_m^k&\quad \mbox{if }\quad m\leq (2n-1)2^k\\\noalign{\medskip}
w_m^kx_{k+1}&\quad \mbox{if }\quad m> (2n-1)2^k.
\end{array}\right.$$
Then, $(iii)$ with $k+1$ in place of $k$ follows by Corollary \ref{CorBfromA}.
\end{proof}

The following lemma is needed in the proof of Proposition \ref{PropDownUp}.

\begin{lemma}\label{LemmaInductionEpsilon}
{\itshape The notation of step 2 in the proof of Proposition
\ref{PropDownUp} is in order. In particular, integers $k$ and $n$ are fixed
with $0\leq k\leq j-1$ and $ 1\leq n\leq 2^{j-k-1}$
and will be dropped from the notation. 
The simplified notation is 
$C(s)=B(s)-\delta P$, where $\delta=\delta_n^k$ is the coefficient of $P:=P^k_{-1}$ 
in item (iii) of Proposition \ref{PropUpDown}. 
Then, there is an algorithm to construct functions 
$ z_{m,\varepsilon}\in T_k$
and integers $d_{m,\varepsilon}$
parametrized by integers $m$ and $\varepsilon$ with 
$  0\leq \varepsilon\leq \delta$ and $(n-1)2^k+1\leq m\leq n2^k$
in such a way that for any $\varepsilon=0,..,\delta$, and for any $s\in\bb{Z}$
the family parametrized
by $m$ and $l$:
\begin{eqnarray}\label{FormBasisL(B-eP)}
 x_0^lz_{m,\varepsilon}\qquad\qquad (n-1)2^k+1\leq m\leq n2^k\qquad
0\leq l\leq
{\textstyle \left \lfloor \frac{s-b_m}{2^{j}}\right \rfloor}
-d_{m,\varepsilon}
\end{eqnarray}
is a basis of $L(B(s)-\varepsilon P)$.
}
\end{lemma}

\begin{proof}
The proof is by induction on $\varepsilon$.
For $\varepsilon=0$, define $ d_{m,0}:= d^k_m$ and $ z_{m,0}:=z^k_m$,
where $d^k_m$ and $z^k_m$ are given in step 2 of the proof of Proposition
\ref{PropDownUp}.

Fix $\varepsilon$ with
$0\leq \varepsilon <\delta$, and suppose
to have already defined $d_{m,\varepsilon}$ and $z_{m,\varepsilon}$
giving a basis of $L(B(s)-\varepsilon P)$ for any
$s$. We will define integers
$d_{m,\varepsilon+1}$ and elements $z_{m,\varepsilon+1}$ that give a basis
of $L(B(s)-(\varepsilon+1) P)$ for any $s$.

Define, for any $m$ and $s$, 
$ l_{m}(s):={\textstyle \left \lfloor \frac{s-b_m}{2^{j}}\right \rfloor}
-d_{m,\varepsilon}$.
By Lemma \ref{LemmaSigma} there is a unique bijection:
$$\sigma: \{i\,|\, 1\leq i\leq 2^k\}\map
\{m\,|\, (n-1)2^k+1\leq m\leq n2^k\}$$
s.t. for any $s\in\bb{Z}$ we have 
$l_{\sigma(1)}(s)\leq l_{\sigma(2)}(s)\leq\, ...\, \leq l_{\sigma(2^k)}(s)$.
Next, we prove the following two facts that will be needed in the sequel:

\noindent
(i) For any $m=(n-1)2^k+1,...,n2^k$ we have $v_P(z_{m,\varepsilon})\geq
\varepsilon$.

\noindent
(ii) There exists $m$ s.t. $v_P(z_{m,\varepsilon})=
\varepsilon$.

Indeed, we will use the fact that formula (\ref{FormBasisL(B-eP)}) is valid for any
$s\in\bb{Z}$ (for the given $\varepsilon$).
In particular, for $s$ big enough we have $l_m(s)\geq 0$
for any $m$. It follows, taking $l=0$ for any $m$, that
$z_{m,\varepsilon}\in L(B(s)-\varepsilon P)$; hence,
$v_P(z_{m,\varepsilon})\geq
\varepsilon$, and (i) is proved.
Since, by Proposition \ref{PropUpDown} $(ii,b)$,
$\mr{deg}\,B(s)\ra +\infty$ as $s\ra +\infty$, then, applying the Riemann-Roch Theorem, 
it follows that
for $s$ big enough we have $L((B(s)-\varepsilon P) -P)\neq
L(B(s)-\varepsilon P)$; thus, it cannot happen that
$v_P(z_{m,\varepsilon})>
\varepsilon$ for any
$m$ (use that $v_P(x_0)\geq 0$, and Remark \ref{RemOrder}),
and item (ii) is proved.\\

The function $t:=1-x_0$ is a local parameter around $P$: 
since  $\delta$ is strictly positive
 then
$k>0$, and we can apply Remark \ref{RemPower}. 

For $z\in T_k$ with $v_P(z)\geq \varepsilon$,
we denote by $z(P)\in K$
 the  order $\varepsilon$ coefficient
in the power series expansion of $z$ around $P$ with respect to $t$: 
$z=z(P)t^\varepsilon+O(t^{\varepsilon+1})$.
Notice that $z(P)=0$ if and only if $v_P(z)>\varepsilon$.
Denote:
$$ I:=\mr{max}\{i\,|\, 1\leq i\leq 2^k\mbox{ and }
v_P(z_{\sigma(i),\varepsilon})=
\varepsilon\}.$$
Notice that $z_{\sigma(I),\varepsilon}(P)\neq 0$.
The constructibility of  $z_{m,\varepsilon}(P)$
follows by Remark \ref{RemPower} and by the fact that the functions $
z_{m,\varepsilon}$ are expressed in terms of the generators $x_i$: in fact, they
are constructed
starting from the explicit functions $z_m^0$ of step 1 of
the proof of Proposition \ref{PropDownUp}.

We are now in the position to define $z_{m,\varepsilon+1}$ and
$d_{m,\varepsilon+1}$. For any $m$ with
$(n-1)2^k+1\leq m\leq n2^k$ define:
$$ z_{m,\varepsilon+1}:=\left\{
\begin{array}{ll}
z_{m,\varepsilon}-\frac{z_{m,\varepsilon}(P)}{z_{\sigma(I),\varepsilon}(P)}
z_{\sigma(I),\varepsilon}
&\quad\mbox{if }\quad m\neq \sigma(I)
\\\noalign{\medskip}
z_{\sigma(I),\varepsilon}(1-x_0)
&\quad\mbox{if }\quad m= \sigma(I)
\end{array}\right.$$
and:
$$ d_{m,\varepsilon+1}:=\left\{
\begin{array}{ll}
d_{m,\varepsilon}
&\quad\mbox{if }\quad m\neq \sigma(I)
\\\noalign{\medskip}
d_{\sigma(I),\varepsilon}+1
&\quad\mbox{if }\quad m= \sigma(I).
\end{array}\right.$$

We are left to prove that for any $s\in\bb{Z}$ the $z_{m,\varepsilon+1}$'s
and the  $d_{m,\varepsilon+1}$'s give a basis of $L(B(s)-(\varepsilon+1))$
as in the statement. We isolate the proof of this fact in Lemma
\ref{LemBasisD-P}.

\end{proof}

The following lemmas and remark are
 needed in the proof of Lemma \ref{LemmaInductionEpsilon}.

\begin{lemma} \label{LemmaSigma}
With the notation of  the proof of
Lemma \ref{LemmaInductionEpsilon}, there exists a unique
bijection:
$$\sigma: \{i\,|\, 1\leq i\leq 2^k\}\map
\{m\,|\, (n-1)2^k+1\leq m\leq n2^k\}$$
s.t., for any $s\in\bb{Z}$, we have 
$l_{\sigma(1)}(s)\leq l_{\sigma(2)}(s)\leq\, ...\, \leq l_{\sigma(2^k)}(s)$.

\end{lemma}

\begin{proof}
We put a lexicographical order on $\bb{Z}\times \bb{Z}$ by declaring:
$ (d,b)\leq (d',b')$ if $d< d'$ or
$(d=d'\mbox{ and } b\leq b').$
Since as $m$ varies the $b_m$'s
are all distinct, then there exists a unique bijection
$\sigma$ as in the statement s.t.:
$ (d_{\sigma(1),\varepsilon}, b_{\sigma(1)})>
(d_{\sigma(2),\varepsilon}, b_{\sigma(2)})>\,
...\, >
(d_{\sigma(2^k),\varepsilon}, b_{\sigma(2^k)})$.
We have to show that such a $\sigma$ has the desired
property. For this it's enough to show that if $(d,b),(d',b')\in\bb{Z}\times\bb{Z}$
with $0\leq b,b'<2^j$ and $(b,d)\geq (b',d')$ then, for all
$s\in\bb{Z}$:
\begin{eqnarray*}\label{FormOrderlms}
 {\textstyle \Big \lfloor \frac{s-b}{2^{j}}\Big \rfloor}
-d\leq
{\textstyle \Big\lfloor \frac{s-{b}'}{2^{j}}\Big \rfloor}
-d'.
\end{eqnarray*}
We leave the easy proof to the reader.

\end{proof}

\begin{remark} \label{RemOrder}
With the notation  of the proof of
Lemma \ref{LemmaInductionEpsilon}, for any $s\in\bb{Z}$ and
for any
$z\in L(B(s)-\varepsilon P)$ we have that $v_P(z)\geq \varepsilon$ and:
$$ z\in L(B(s)-(\varepsilon+1)P) \quad\iff\quad v_P(z)>\varepsilon
\quad \iff\quad z(P)=0.$$

\end{remark}

\begin{lemma}\label{LemBasisD-P}
With the notation  of the proof of
Lemma \ref{LemmaInductionEpsilon},
for any $s\in\bb{Z}$
the family parametrized
by $m$ and $l$:
$$
 x_0^lz_{m,\varepsilon+1}\qquad\qquad (n-1)2^k+1\leq m\leq n2^k\qquad
0\leq l\leq
{\textstyle \left \lfloor \frac{s-b_m}{2^{j}}\right \rfloor}
-d_{m,\varepsilon+1}
$$
is a basis of $L(B(s)-(\varepsilon+1)P)$.

\end{lemma}

\begin{proof}
Fix $s\in\bb{Z}$. We simplify the notation by writing, for any $i$ with
$1\leq i\leq 2^k$:
$$z_i=z_{\sigma(i),\varepsilon},\quad
d_i=d_{\sigma(i),\varepsilon},\quad\bar{z}_i=z_{\sigma(i),\varepsilon+1},\quad
\bar{d}_i=d_{\sigma(i),\varepsilon+1}$$
and:
$$ l_i=l_{\sigma(i)}(s),\qquad \bar{l}_i=\bar{l}_{\sigma(i)}(s)\qquad\qquad
D=B(s)-\varepsilon P.$$
With the new notation we have that $l_1\leq l_2\leq\,...\,\leq l_{2^k}$, and:
$$\bar{z}_i=\left\{
\begin{array}{ll}
z_i-\frac{z_i(P)}{z_I(P)}z_I&\quad\mbox{if}\quad i\neq I\\\noalign{\medskip}
z_I(1-x_0)&\quad\mbox{if}\quad i=I
\end{array}\right.\qquad\qquad \bar{l}_i=\left\{
\begin{array}{ll}
l_i&\quad\mbox{if}\quad i\neq I\\\noalign{\medskip}
l_I-1&\quad\mbox{if}\quad i=I.
\end{array}\right.$$
Recall that for any $i$ we have $v_P(z_i)\geq \varepsilon$, and $I$ is the biggest
index with $v_P(z_I)=\varepsilon$.

What is left to prove is the following:
if the family
parametrized by $i$ and $l$: $ x_0^lz_i$ with $1\leq i\leq 2^k$ and $0\leq l\leq l_i$
is a basis of $L(D)$, then the family
$ x_0^l\bar{z}_i$ with $ 0\leq  i\leq 2^k$ and $ 0\leq l\leq \bar{l}_i$
is a basis of $L(D-P)$.\\

If $L(D-P)=L(D)$, we start by observing that if $i$ is such that $l_i\geq 0$
then  $v_P(z_i)>\varepsilon $ and $z_i(P)=0$: this is a consequence
of the fact that $z_i\in L(D)=L(D-P)$ and Remark \ref{RemOrder}.
Notice that such values of $i$ are
exactly the ones that contribute to the basis of $L(D)$, and that,
in particular, $l_I<0$ since $v_P(z_I)=\varepsilon $. It follows that,
for $i$ with $l_i\geq 0$, we have $\bar{z}_i=z_i$ and $\bar{l}_i=l_i$.
We conclude that the family $x_0^l\bar{z_i}$ is exactly the
same as the family $x_0^lz_i$, and we are done.
\\

If $L(D-P)\varsubsetneq L(D)$, we start by observing that the codimension is $1$ and
there exists $i$ with $l_i\geq 0$ and $v_P(z_i)=\varepsilon$. Indeed,
if for any $i$ with $l_i\geq 0$ we had $v_P(z_i)>\varepsilon$,
then, by Remark \ref{RemOrder}, we would have that any element $x_0^lz_i$
that contributes to the basis of $L(D)$ would belong to $L(D-P)$, which
would imply $L(D-P)=L(D)$, contrary to the assumption. By the maximality
of $I$ we have that $l_I\geq 0$. In particular, the family $x_0^l\bar{z}_i$
has one element less than the family $x_0^lz_i$. We are left to prove that
all the elements of $x^l_0\bar{z}_i$ belong to $L(D-P)$ and that they
are linearly independent over $K$.\\

Let's prove that for any $i$ and for any $l$ with $0\leq l\leq \bar{l}_i$ then
$x^l_0\bar{z}_i\in L(D-P)$.

First, consider the case $i=I$. Then $x_0^l\bar{z}_I=x_0^lz_I-
x_0^{l+1}z_I$. Since $\bar{l}_I=l_I-1$ and $0\leq l\leq \bar{l}_I$ then
both $x_0^lz_I$ and $x_0^{l+1}z_I$ belongs to $L(D)$, hence
$x_0^l\bar{z}_I$ belongs as well. Since $v_P(x_0^l\bar{z}_I)=
v_P(x_0^lz_I(1-x_0))=v_P(z_I)+v_P(1-x_0)=\varepsilon+1>\varepsilon$,
then we can apply Remark \ref{RemOrder} to conclude
$x_0^l\bar{z}_I\in L(D-P)$.

Second, consider the case in which $i\neq I$ and $v_P(z_i)>\varepsilon$.
Then $z_i(P)=0$.  It follows that $x_0^l\bar{z}_i=x_0^lz_i$. Now,
$x_0^lz_i\in L(D-P)$ by
$x_0^lz_i\in L(D)$ since $0\leq l\leq
\bar{l}_i=l_i$, and by application of Remark \ref{RemOrder}.

Third, consider the case $i\neq I$ and $v_P(z_i)=\varepsilon$. Then
$i<I$ and $\bar{l}_i=l_i\leq l_I$. It follows that for $0\leq l\leq \bar{l}_i$
we have $x_0^l\bar{z}_i=x_0^lz_i-(z_i(P)/z_I(P))x_0^lz_I \in L(D)$
since both $x_0^lz_i$ and $x_0^lz_I$ belongs to $L(D)$. Moreover,
it's clear that $\bar{z}_i(P)=0$ so that
$v_P(x_0^l\bar{z}_i)=v_P(\bar{z}_i)>\varepsilon$, and we can apply
Remark \ref{RemOrder} to conclude $x_0^l\bar{z}_i\in L(D-P)$.
\\

Now, let  
$\alpha_{i,l}\in K$ s.t. $ \sum_{i=1}^{2^k}\sum_{l=0}^{\bar{l}_i}\alpha_{i,l}x_0^l\bar{z}_i=0$.
Write  the set $\cl{I}:=\{i\,| \, 1\leq i\leq 2^k\}$ as a disjoint union:
$$ \cl{I}=\{I\}\cup \cl{I}_0\cup\cl{I}_1 $$
where $\cl{I}_0:=\{i\in\cl{I}\,|\, i\neq I\quad\mbox{and}\quad v_P(z_i)
=\varepsilon\}$ and $\cl{I}_1:=\{i\in\cl{I}\,|\, v_P(z_i)
>\varepsilon\}$, and split the summation:
$$ \sum_{i=1}^{2^k}\sum_{l=0}^{\bar{l}_i}\alpha_{i,l}x_0^l\bar{z}_i=
\sum_{i\in\cl{I}_0}\sum_{l=0}^{l_i}\alpha_{i,l}x_0^l
\left(z_i-{\scriptstyle\frac{z_i(P)}{z_I(P)}}z_I\right)+
\sum_{i\in\cl{I}_1}\sum_{l=0}^{l_i}\alpha_{i,l}x_0^lz_i+
\sum_{l=0}^{l_I-1}\alpha_{I,l}(x_0^lz_I-x_0^{l+1}z_I)
$$
where we used the definitions of
$\bar{z}_i$
and of $\bar{l}_i$.
(We are using that for $i\in\cl{I}_1$ we have
$\bar{z}_i=z_i$.) For any $i\in\cl{I}$, all the $x_0^lz_i$ appearing in the
sum above are
with $0\leq l\leq l_i$. (For the $x_0^lz_I$ in the first summation in the
right-hand side,
we are using that for $i\in\cl{I}_0$ we have $l_i\leq l_I$, by
the maximality of $I$.) Since such elements $x_0^lz_i$ are linearly independent,
then it follows that for $i\neq I$ we have $\alpha_{i,l}=0$.
It follows that $\sum_{l=0}^{l_I}\left(\alpha_{I,l}-\alpha_{I,l-1}\right)x_0^lz_I=0$,
where in case $l_I\geq 0$ we put $\alpha_{I,-1}:=\alpha_{I,l_I}:=0$.
(In case $l_I<0$ there are no coefficients $\alpha_{I,l}$ at all,
and we are done.) By the linear independence of the $x_0^lz_I$'s, it
follows that $\alpha_{I,l-1}=\alpha_{I,l}$ for any $l$. Then,
by induction on $l$, all the coefficient $\alpha_{I,l}$ vanish, and
the proof of the linear independence is complete.
\\
\end{proof}


\subsection{Weierstrass semigroups at $P_\infty^j$ up to level 8}\label{SecSemigroup}

As an application of Theorems \ref{ThBasisLPinf} and \ref{ThWeierLPinf}
we implemented an algorithm in Scilab (www.scilab.org) to compute the Weierstrass
semigroups $H^j:= H(P_\infty^j)$ at $P^j_\infty$ for any $0\leq j\leq 8$.

We list the elements of the semigroup $H^j$ by making a list of the
non-gap intervals: $a_1$-$b_1$, $a_2$-$b_2$, ..., $a_{n}$-$b_n$, $a_{n+1}$-$\infty$.
This means that:
$$H^j=\bigcup_{k=1}^n [a_k,b_k]\cup [a_{n+1},\infty).$$
If $a_k=b_k$ we write $a_k$-$b_k$ as $a_k$. In the following table $g_j$ denotes the genus
of $T_j/K$.\\

$$\begin{array}{|c|c|l|}\hline
j&g_j& H^j\\\hline
0& 0&0\mbox{-}\infty\\\hline
1&1&0;\, 2\mbox{-}\infty\\\hline
2&3&0;\, 3\mbox{-}4;\, 6\mbox{-}\infty\\\hline
3&9&0;\,6;\,8;\,11\mbox{-}12;\, 14\mbox{-}\infty\\\hline
4&21&0;\,12;\,15\mbox{-}16;\,22\mbox{-}24;\,27\mbox{-}32;\, 34\mbox{-}\infty\\\hline
5&49&
0;\,24;\,30\mbox{-}32;\,44;\,46\mbox{-}48;\,
53\mbox{-}56;\,58\mbox{-}64;\,
68\mbox{-}72;\, 74\mbox{-}\infty\\\hline
6&105&
0;\,48;\,60;\,62\mbox{-}64;\,88;\,92;\,94\mbox{-}96;\,
103;\,106\mbox{-}112;\\
&&
115\mbox{-}128;\,
135\mbox{-}136;\,138\mbox{-}144;\,147\mbox{-}\infty\\\hline
7&225&
0;\,96;\,120;\,124;\,126\mbox{-}128;\,176;\,
184;\,188;\,190\mbox{-}192;\,206;\\
&&
212\mbox{-}216;\,
218;\,220\mbox{-}224;\,230\mbox{-}232;\,234\mbox{-}240;\,
242\mbox{-}256;\\
&&
263;\,269\mbox{-}272;\,276\mbox{-}280;\,
282\mbox{-}288;\,291;\,293\mbox{-}\infty\\\hline
\end{array}$$

\noindent
For $j=8$ we have $g_8=465$ and the non-gap intervals are:
$$\begin{array}{|c|c|c|c|c|c|c|c|}
\hline
0&192&240&248&252&254\mbox{-}256&352&368\\
\hline
376&380&382\mbox{-}384&412&423\mbox{-}424&426&428&430\\
\hline
432&436&439\mbox{-}440&442&444\mbox{-}448&459\mbox{-}464&467\mbox{-}472&474\mbox{-}480\\
\hline
483\mbox{-}484&486\mbox{-}512&519&526\mbox{-}527&533&535&538\mbox{-}540&542\mbox{-}544\\
\hline
547&549&551\mbox{-}552&554\mbox{-}561&563\mbox{-}576&579&581\mbox{-}583&
585\mbox{-}\infty\\
\hline
\end{array}$$

\bibliographystyle{amsalpha}

\end{document}